\documentclass[a4paper,leqno,11pt]{article}
\usepackage{amsmath, amsthm, amssymb}
\usepackage{mathrsfs}

\newtheorem{thm}{Theorem}[section]
\newtheorem{cor}[thm]{Corollary}
\newtheorem{lem}[thm]{Lemma}
\newtheorem{prop}[thm]{Proposition}

\theoremstyle{definition}

\newtheorem{rem}[thm]{Remark}

\makeatletter
\@addtoreset{equation}{section} 

\makeatother

\begin{document}

%
%

\newcommand{\BC}{\mathbb{C}}
\newcommand{\BR}{\mathbb{R}}
\newcommand{\BN}{\mathbb{N}}
\newcommand{\BZ}{\mathbb{Z}}

%
%

\title{Uniqueness of the solution of nonlinear singular 
first order partial differential equations}

\author{{} \\
Hidetoshi \textsc{TAHARA}\footnote{Department of Information and 
        Communication Sciences, Sophia University, Tokyo 
        102-8554, Japan.
        e-mail: \texttt{h-tahara@sophia.ac.jp}}}

\date{}

\maketitle

\begin{abstract}
    This paper deals with nonlinear singular partial differential
equations of the form 
$t \partial u/\partial t=F(t,x,u,\partial u/\partial x)$ with
independent variables $(t,x) \in \BR \times \BC$, where 
$F(t,x,u,v)$ is a function continuous in $t$ and holomorphic in 
the other variables. Under a very weak assumption we show the 
uniqueness of the solution of this equation. The results are 
applied to the problem of analytic continuation of local holomorphic
solutions of equations of this type.
\end{abstract}



\renewcommand{\thefootnote}{\fnsymbol{footnote}}

\footnote[0]{
      This work was supported by JSPS KAKENHI Grant Number 
        JP15K04966.}

%
%
%
%

\section{Introduction}\label{section1}
%

   To investigate the uniqueness of the solution is one of the most
important problems in the theory of partial differential equations,
and there are many references in various situations.
In this paper, we consider the case of first order nonlinear singular 
partial differential equations (\ref{1.1}) given below, and show uniqueness 
results by a method quite similar to the Cauchy's 
characteristic method. 
\par
   Let $t \in \BR$, $x \in \BC$, $u \in \BC$ and $v \in \BC$ be 
the variables.  For $r>0$ we write $D_r=\{z \in \BC \,;\, |z|<r \}$
where $z$ represents $x$, $u$ or $v$. Let $T_0>0$, $R_0>0$, $\rho_0>0$, 
and set
$\Omega=\{(t,x,u,v) \in [0,T_0] \times D_{R_0} \times D_{\rho_0} 
\times D_{\rho_0} \}$.
\par
   Let $F(t,x,u,v)$ be a function on $\Omega$. In this paper, we 
consider the equation

\begin{equation}\label{1.1}
   t \,\frac{\partial u}{\partial t}
   = F \Bigl( t,x,u, \frac{\partial u}{\partial x} \Bigr)
\end{equation}
under the following assumptions:
\par
\medskip
   ${\rm A}_1)$ $F(t,x,u,v)$ is a continuous function on 
$\Omega$ which is holomorphic in the variable 
$(x,u,v) \in D_{R_0} \times D_{\rho_0} \times D_{\rho_0}$ for any 
fixed $t$.
\par
   ${\rm A}_2)$ There is a weight function $\mu(t)$ on
$(0,T_0]$ satisfying the following:
\begin{align*}
    &\sup_{x \in D_{R_0}}|F(t,x,0,0)|=O(\mu(t)) 
                \quad \mbox{(as $t \longrightarrow +0$)}, \\
    &\Bigl|\frac{\partial F}{\partial v}
           (t,0,0,0) \Bigr|=O(\mu(t)) 
           \quad \mbox{(as $t \longrightarrow +0$)}.
\end{align*}
\par
    Here, a weight function $\mu(t)$ on $(0,T_0]$ means that $\mu(t)$
is a positive-valued continuous function on $(0,T_0]$ which is 
increasing in $t$ and satisfies
\[
      \int_0^{T_0} \frac{\mu(s)}{s} ds < \infty.
\]
By this condition, we have $\mu(t) \longrightarrow 0$ 
(as $t \longrightarrow +0$). 
\par
   By ${\rm A}_2)$ we can express 
$(\partial F/\partial v)(t,x,0,0)$ in the form
\[
   \frac{\partial F}{\partial v}(t,x,0,0)
    = b(t)+x^{p+1}c(t,x)
\]
where $b(t)$ is a continuous function on $[0,T_0]$ satisfying
$b(t)=O(\mu(t))$ (as $t \longrightarrow +0$), 
$c(t,x)$ is a continuous function on 
$[0,T_0] \times D_{R_0}$ that is holomorphic in $x$, 
and $p \in \{0,1,2,\ldots \}$. Then, we can divide our situation into 
the following three cases:
\par
\medskip \qquad
   Case 1. $c(t,x) \equiv 0$ on $[0,T_0] \times D_{R_0}$,
\par  \qquad
   Case 2. $p=0$ and $c(t,0) \not\equiv 0$ on $[0,T_0]$,
\par  \qquad
   Case 3. $p \geq 1$ and $c(t,0) \not\equiv 0$ on $[0,T_0]$.
\par
\medskip
\noindent
In Case 1, equation (\ref{1.1}) is a generalization of Briot-Bouquet's
ordinary differential equations (in Briot-Bouquet \cite{briot})
to partial differential equations, and this type of equations was
studied by Baouendi-Goulaouic \cite{BG}, G\'{e}rard-Tahara \cite{gt1}, 
Yamazawa \cite{yamazawa}, Koike \cite{koike} and 
Lope-Roque-Tahara \cite{LRT}.
In Case 2, equation (\ref{1.1}) has a regular singularity at $x=0$, 
and this type of equations was studied by Chen-Tahara \cite{chen1} 
and Bacani-Tahara \cite{BT1}.
In Case 3, equation (\ref{1.1}) has an irregular singularity at $x=0$,
and this type of equations was studied by Chen-Luo-Zhang \cite{CLZ},
Luo-Chen-Zhang \cite{LCZ} and Bacani-Tahara \cite{BT2}.
In these papers, mainly the solvability (or the unique solvability)
of equation (1.1) is discussed.
\par
   As to the uniqueness of the solution, we know some results:
in Case 1 we have a result in Tahara \cite{unique} under the assumption:
$u(t,x)=O(\mu(t)^{\epsilon})$ (as $t \longrightarrow +0$) for some 
$\epsilon>0$, and in Case 2 we have a result in Tahara \cite{unique2} 
under the assumption: $u(t,x)=O(|t|^\epsilon)$ (as $t \longrightarrow +0$) 
for some $\epsilon>0$.
\par
   In this paper, we will show the uniqueness of the solution in 
each case under a much weaker assumption like
\[
     \varlimsup_{R \to +0} \, \biggl[ \, 
          \lim_{\sigma \to +0} \, \Bigl( \frac{1}{R^2}
       \sup_{(0,\sigma) \times D_R}|u(t,x)| \Bigr) \, \biggr]
       \, = \, 0.
\]

%
%
%

\section{Analysis in Case 1}\label{section2}
%

   Let us consider Case 1 in a little bit general setting.
We consider equation (\ref{1.1}) under the following assumptions:
\begin{align}
    &\sup_{x \in D_{R_0}}|F(t,x,0,0)|=O(\mu(t)) 
                \quad \mbox{(as $t \longrightarrow +0$)},
          \label{2.1}\\
    &\sup_{x \in D_{R_0}} \Bigl|\frac{\partial F}{\partial v}
           (t,x,0,0) \Bigr|=O(\mu(t)) 
           \quad \mbox{(as $t \longrightarrow +0$)}. 
          \label{2.2}
\end{align}
\par
   As to the existence of a solution, we know a unique solvability 
result in a certain function space. To state the existence result, 
let us prepare some notations.  We set
\[
      \varphi(t)= \int_0^t \frac{\mu(s)}{s} ds, \quad 0<t \leq T_0.
\]
This is also an increasing function on $(0,T_0]$ and we have
$\varphi(t) \longrightarrow 0$ (as $t \longrightarrow +0$). For $T>0$,
$R>0$ and $r>0$ we set
\[
    W_{T,R,r}=\{(t,x) \in [0,T] \times \BC \,;\, 
       \varphi(t)/r + |x|<R \}.
\]
For $W=W_{T,R,r}$, we denote by ${\mathscr X}_0(W)$ the set of all 
functions in $C^0(W)$ that are holomorphic in $x$ for any fixed $t$, 
and by ${\mathscr X}_1(W)$ the set of all functions in 
$C^1(W \cap \{ t>0 \}) \cap C^0(W)$ that are also holomorphic in $x$ 
for any fixed $t$. We set
\[
     \lambda(t,x)= \frac{\partial F}{\partial u}(t,x,0,0).
\]
By [Theorem 1.1 (with $\alpha=1$) in Lope-Roque-Tahara \cite{LRT}] 
we have

\begin{thm}\label{Theorem2.1}
    Suppose the conditions {\rm (\ref{2.1})} and {\rm (\ref{2.2})}.
If ${\rm Re}\lambda (0,0)<0$ holds, there are $T>0$, $R>0$ and $r>0$ 
such that equation {\rm (\ref{1.1})} has a unique solution
$u_0(t,x) \in {\mathscr X}_1(W_{T,R,r})$ satisfying
\[
    |u_0(t,x)| \leq M\mu(t) \quad \mbox{and} \quad 
    \Bigl| \frac{\partial u_0}{\partial x}(t,x) \Bigr| 
       \leq M\mu(t)
\]
on $W_{T,R,r}$ for some $M>0$.
\end{thm}

%
%
%

\subsection{Uniqueness result in Case 1}\label{subsection2.1}
%

   For $T>0$ and $R>0$ we denote by ${\mathscr X}_1((0,T) \times D_R)$
the set of all functions in $C^1((0,T) \times D_R)$ that are 
holomorphic in the variable $x \in D_R$ for any fixed $t$. 
\par
   The following theorem is the main result of this section.

\begin{thm}\label{Theorem2.2} 
     Suppose the conditions {\rm (\ref{2.1})}, {\rm (\ref{2.2})} and
${\rm Re}\lambda(0,0)<0$.
Let $u(t,x) \in {\mathscr X}_1((0,T) \times D_R)$ be a solution of 
{\rm (\ref{1.1})} with $T>0$ and $R>0$. If $u(t,x)$ satisfies
\begin{equation}\label{2.3}
     \varlimsup_{R \to +0} \, \biggl[ \, 
          \lim_{\sigma \to +0} \, \Bigl( \frac{1}{R^2}
       \sup_{(0,\sigma) \times D_R}|u(t,x)| \Bigr) \, \biggr]
       \, = \, 0, 
\end{equation}
we have $u(t,x)=u_0(t,x)$ on $(0,T_1) \times D_{R_1}$ for 
some $T_1>0$ and $R_1>0$, where $u_0(t,x)$ is the solution 
of {\rm (\ref{1.1})} obtained in Theorem \ref{Theorem2.1}.
\end{thm}

   If 
\begin{equation}\label{2.4}
     \lim_{t \to +0} \,\Bigl(\sup_{x \in D_R}|u(t,x)| \Bigr)\,
       = \, 0  
\end{equation}
holds for some $R>0$ we have (\ref{2.3}), and so we have

\begin{cor}\label{Corollary2.3}
    Suppose the conditions {\rm (\ref{2.1})}, {\rm (\ref{2.2})} and
${\rm Re}\lambda(0,0)<0$.
If a solution $u(t,x) \in {\mathscr X}_1((0,T) \times D_R)$ of 
{\rm (\ref{1.1})} satisfies {\rm (\ref{2.4})}, we have 
$u(t,x)=u_0(t,x)$ on $(0,T_1) \times D_{R_1}$ for some $T_1>0$ and 
$R_1>0$.
\end{cor}

    If a solution $u(t,x)$ satisfies 
\begin{equation}\label{2.5}
     \sup_{x \in D_R}|u(t,x)|=O(\mu(t)^{\epsilon})
     \quad \mbox{(as $t \longrightarrow +0$)}
\end{equation}
for some $\epsilon >0$, we can apply a result in Tahara \cite{unique}.
We note that the condition (\ref{2.3}) is much weaker than (\ref{2.5}).
In \cite{unique} higher order equations are dealt with, but it is 
unclear whether we can generalize Theorem \ref{Theorem2.2} to higher 
order case.

\begin{rem}\label{Remark2.4}
    (1) In the case ${\rm Re}\lambda(0,0)>0$, we can
give many examples in holomorphic category such that the equation has 
many solutions satisfying (\ref{2.4}). Therefore, the uniqueness of the
solution is not valid in general. See \cite {gt1} and \cite{yamazawa}.
\par
   (2) In the case ${\rm Re}\lambda(0,0)=0$, we have the
following counter example: the equation
\[
   t \, \frac{\partial u}{\partial t}
    = u \Bigl(\frac{\partial u}{\partial x} \Bigr)^k
     \quad \mbox{($k \in \{1,2,\ldots \}$)}
\]
has a trivial solution $u \equiv 0$ and a family of nontrivial 
solutions
\[
    u = \Bigl( \dfrac{1}{k} \Bigr)^{1/k}
          \dfrac{x + \alpha}{( c - \mbox{log}\,t )^{1/k}}
\]
with arbitrary constants $\alpha$ and $c$. These solutions satisfy 
(2.4).
\par
   (3) The following example shows that the assumption (\ref{2.3}) 
is reasonable: the equation
\[
     t \frac{\partial u}{\partial t}
     = -u+ \Bigl(\frac{\partial u}{\partial x} \Bigr)^2
\]
has a trivial solution $u \equiv 0$ and a nontrivial solution 
$u=x^2/4$.  We note that for $u=x^2/4$ we have
\[
     \varlimsup_{R \to +0} \, \biggl[ \, 
          \lim_{\sigma \to +0} \, \Bigl( \frac{1}{R^2}
       \sup_{(0,\sigma) \times D_R}|u(t,x)| \Bigr) \, \biggr]
       =  \frac{1}{4}.
\]
\end{rem}

%
%

\subsection{Proof of Theorem 2.2}\label{subsection2.2}

   Let $u_0(t,x)$ be the unique solution of (\ref{1.1}) obtained in 
Theroem \ref{Theorem2.1}. Set $v_0(t,x)=(\partial u_0/\partial x)(t,x)$.
Then, by setting $w=u-u_0$, our equation (\ref{1.1}) is reduced to an 
equation with respect to $w=w(t,x)$:
\begin{equation}\label{2.6}
    t \,\frac{\partial w}{\partial t}
     = H \Bigl(t,x,w,\frac{\partial w}{\partial x} \Bigr)
\end{equation}
where
\begin{align*}
    H(t,x,w,q)= &F(t,x,w+u_0(t,x), q+v_0(t,x)) \\
        &\qquad \qquad \qquad   - F(t,x,u_0(t,x), v_0(t,x)).
\end{align*}
\par
    For $\Omega^*=\{t,x,u,v) \in 
[0,\sigma^*] \times D_{R_0^*} \times D_{\rho_0^*}
\times D_{\rho_0^*} \}$ we denote by ${\mathscr X}_0(\Omega^*)$ the 
set of all functions in $C^0(\Omega^*)$ that are holomorphic in the 
variable $(x,w,q)$ for any fixed $t$.
\par
   Then, we may suppose that $H(t,x,w,q)$ belongs to 
${\mathscr X}_0(\Omega^*)$ for sufficiently small $\sigma^*>0$,
$R_0^*>0$ and $\rho_0^*>0$.  It is easy to see that $H(t,x,w,q)$ is 
expressed in the form
\[
    H(t,x,w,q)= \lambda(t,x)w + a_1(t,x,w,q)w + b_1(t,x,w,q)q
\]
for some functions $a_1(t,x,w,q) \in {\mathscr X}_0(\Omega^*)$ 
and $b_1(t,x,w,q) \in {\mathscr X}_0(\Omega^*)$ satisfying 
\begin{align*}
    &\sup_{x \in D_{R_0^*}}|a_1(t,x,0,0)|=O(\mu(t)) \quad
       \mbox{(as $t \longrightarrow +0$)}, \\
    &\sup_{x \in D_{R_0^*}}|b_1(t,x,0,0)|=O(\mu(t)) \quad
       \mbox{(as $t \longrightarrow +0$)}.
\end{align*}
To get Theorem \ref{Theorem2.2} it is sufficient to show the following 
result.

\begin{prop}\label{Proposition2.5} 
    Suppose ${\rm Re}\lambda(0,0)<0$.
Let $w(t,x) \in {\mathscr X}_1((0,\sigma_0) \times D_{R_0})$ be a 
solution of {\rm (\ref{2.6})} with $\sigma_0>0$ and $R_0>0$. If 
$w(t,x)$ satisfies
\begin{equation}\label{2.7}
     \varlimsup_{R \to +0} \, \biggl[ \, 
          \lim_{\sigma \to +0} \, \Bigl( \frac{1}{R^2}
       \sup_{(0,\sigma) \times D_R}|w(t,x)| \Bigr) \, \biggr]
       \, = \, 0,
\end{equation}
we have $w(t,x)=0$ on $(0, \sigma) \times D_{\delta}$ for some 
$\sigma>0$ and $\delta>0$.
\end{prop}

\begin{proof}
    Let us prove this step by step.

\par
\medskip
   {\bf Step 1.} Since $\sigma^*>0$ and $R_0^*>0$ are sufficiently 
small, we may suppose that there is an $a>0$ satisfying
\[
     {\rm Re}\lambda(t,x) < -2a \quad
           \mbox{on $[0,\sigma^*] \times D_{R_0^*}$}.
\]
Since $a_1(t,x,0,0)=O(\mu(t))$ and $b_1(t,x,0,0)=O(\mu(t))$ hold, 
we have the estimates
\begin{align*}
   &|a_1(t,x,w,q)| \leq A_0 \mu(t) + A_1|w|+A_2|q| \quad
            \mbox{on $\Omega^*$}, \\
   &|b_1(t,x,w,q)| \leq B_0 \mu(t) + B_1|w|+B_2|q| \quad
            \mbox{on $\Omega^*$}
\end{align*}
for some $A_i>0$ ($i=0,1,2$) and $B_i>0$ ($i=0,1,2$).

\par
\medskip
   {\bf Step 2.} Let $w(t,x) \in {\mathscr X}_1((0,\sigma_0) 
                                            \times D_{R_0})$ be a 
solution of (\ref{2.6}) for some $0<\sigma_0<\sigma^*$ 
and $0<R_0<R_0^*$. We suppose that $w(t,x)$ satisfies (\ref{2.7}).
We set $q(t,x)= (\partial w/\partial x)(t,x)$ and
\begin{align*}
    &a(t,x)=a_1(t,x,w(t,x),q(t,x)), \\
    &b(t,x)=b_1(t,x,w(t,x),q(t,x)):
\end{align*}
these are functions belonging to 
${\mathscr X}_0((0,\sigma_0) \times D_{R_0})$.  Then, by (\ref{2.6}) 
we see that $w(t,x)$ satisfies the following linear partial 
differential equation:
\begin{equation}\label{2.8}
     t \frac{\partial w}{\partial t}
            - b(t,x) \frac{\partial w}{\partial x}
     = (\lambda(t,x)+a(t,x))w. 
\end{equation}
By applying $\partial/\partial x$ to (\ref{2.8}) we have
\begin{equation}\label{2.9}
    t \frac{\partial q}{\partial t}
          - b(t,x) \frac{\partial q}{\partial x} 
    = \gamma(t,x)w + (\lambda(t,x)+a(t,x)+\ell(t,x))q,
\end{equation}
where 
\begin{align*}
    &\gamma(t,x)= (\partial \lambda/\partial x)(t,x)
         + (\partial a/\partial x)(t,x), \\
    &\ell(t,x)= (\partial b/\partial x)(t,x):
\end{align*}
these are also functions belonging to
${\mathscr X}_0((0,\sigma_0) \times D_{R_0})$.
For $0<\sigma<\sigma_0$ and $0<R<R_0$ we set
\[
   A = \!\!\sup_{(0,\sigma) \times D_{R}} \!\!|a(t,x)|, \quad
   \Gamma = \!\! \sup_{(0,\sigma) \times D_{R}} \!\!|\gamma(t,x)|, 
   \quad
   L =  \!\!\sup_{(0,\sigma) \times D_{R}} \!\!|\ell(t,x)|.
\]
We set also
\[
    r_1= \!\!\sup_{(0,\sigma) \times D_R}\!\!|w(t,x)|, \quad
    r_2= \!\!\sup_{(0,\sigma) \times D_R}\!\!|q(t,x)|.
\]

\begin{lem}\label{Lemma2.6}
    By taking $\sigma>0$ and $R>0$ sufficiently small we have the 
conditions $A+L < a$, and
\[
      B_0 \varphi(\sigma) + \Bigl( \frac{B_1}{a} 
        + \frac{B_2\Gamma}{a^2} \Bigr)\,r_1
         + \frac{B_2}{a} \,r_2 < \frac{R}{2}.
\]
\end{lem}

\begin{proof}
   By (\ref{2.7}) we have
\begin{equation}\label{2.10}
    \lim_{\sigma \to +0} 
         \sup_{(0,\sigma) \times D_R}|w(t,x)|
        = o(R^2) \quad \mbox{(as $R \longrightarrow +0$)}.  
\end{equation}
By applying Cauchy's integral formula in $x$ to (\ref{2.10}) we have
\begin{equation}\label{2.11}
     \lim_{\sigma \to +0} 
         \sup_{(0,\sigma) \times D_R}|q(t,x)|
        = o(R) \quad \mbox{(as $R \longrightarrow +0$)}.
\end{equation}
Since $|a_1(t,x,w,q)| \leq A_0\mu(t) + A_1|w|+A_2|q|$ and
$|b_1(t,x,w,q)| \leq B_0 \mu(t) + B_1|w|+B_2|q|$ are known, by 
(\ref{2.10}) and (\ref{2.11}) we have
\begin{align*}
   &\lim_{\sigma \to +0} 
         \sup_{(0,\sigma) \times D_R}|a(t,x)|= o(R) \quad
                  \mbox{(as $R \longrightarrow +0$)}, \\
   &\lim_{\sigma \to +0} 
         \sup_{(0,\sigma) \times D_R}|b(t,x)|= o(R) \quad
                   \mbox{(as $R \longrightarrow +0$)}, \\
   &\lim_{\sigma \to +0} 
         \sup_{(0,\sigma) \times D_R}
         |(\partial b/\partial x)(t,x)|= o(1) \quad
                   \mbox{(as $R \longrightarrow +0$)}.
\end{align*}
Therefore, by taking $\sigma>0$ and $R>0$ sufficiently small, 
the numbers $A$, $L$, $r_1/R$ and $r_2/R$ will be as small as 
possible. This proves Lemma \ref{Lemma2.6}. 
\end{proof}

\par
\medskip
   {\bf Step 3.} Let $\sigma>0$ and $R>0$ be as in Lemma \ref{Lemma2.6}. 
Take any $t_0 \in (0,\sigma)$ and $\xi \in D_R$; for a while we fix 
them.
\par
    Let us consider the initial value problem
\begin{equation}\label{2.12}
   t \,\frac{dx}{dt} = -b(t,x), \quad x(t_0)=\xi. 
\end{equation}
Here, we regard $b(t,x)$ as a function in
${\mathscr X}_0((0,\sigma) \times D_R)$.  Let $x(t)$ be the unique 
solution in a neighborhood of $t=t_0$. Let $(t_{\xi},t_0]$ be the 
maximal interval of the existence of this solution. Set
\[
      w^*(t)=w(t,x(t)), \quad q^*(t)=q(t,x(t)).
\]
Then, by (\ref{2.8}) and (\ref{2.9}) we have
\begin{equation}\label{2.13}
    t \, \dfrac{dw^*(t)}{dt}= (\lambda(t,x(t))+ a(t,x(t)))w^*(t), 
               \quad w^*(t_0)=w(t_0,\xi) 
\end{equation}
on $(t_{\xi},t_0]$, and
\begin{align}
    &t \, \dfrac{dq^*(t)}{dt}
    = \gamma(t,x(t))w^*(t)+ (\lambda(t,x(t))
              + a(t,x(t))+ \ell(t,x(t)))q^*(t),
      \label{2.14} \\
     &q^*(t_0)=q(t_0,\xi) \notag
\end{align}
on $(t_{\xi},t_0]$.

\begin{lem}\label{Lemma2.7}
    Under the above situation, we have the following estimates for 
any $(t_1,\tau)$ satisfying $t_\xi<t_1<\tau \leq t_0$:
\begin{align}
    &|w^*(\tau)| \leq \Bigl( \frac{t_1}{\tau} \Bigr)^{a}|w^*(t_1)|, 
               \label{2.15}\\
    &|q^*(\tau)| \leq \Bigl( \frac{t_1}{\tau} \Bigr)^a
       \bigl( \Gamma|w^*(t_1)| \log(\tau/t_1) + |q^*(t_1)| \bigr).
            \label{2.16}
\end{align}
\end{lem}

\begin{proof}
    Let $t_\xi<t_1<\tau \leq t_0$: set
\[
     \phi(t)=\exp \Bigl[ \int_t^{\tau}
          \frac{(\lambda(s,x(s))+a(s,x(s)))}{s}ds 
          \Bigr], \quad t_1 \leq t \leq \tau.
\]
Since ${\rm Re}(\lambda(s,x(s))+a(s,x(s)))<-2a+A <-a$ we have
\begin{align*}
     |\phi(t)| &\leq \exp \Bigl[ \int_t^{\tau}
        \frac{{\rm Re}(\lambda(s,x(s))+a(s,x(s)))}{s}ds \Bigr] \\
     &\leq \exp \Bigl[ \int_t^{\tau}
                     \frac{-a}{s}ds \Bigr] 
     = \Bigl( \frac{t}{\tau} \Bigr)^{a}, 
                \quad t_1 \leq t \leq \tau.
\end{align*}
Let us show (\ref{2.15}). By (\ref{2.13}) we have
\[
    \frac{d}{dt} (w^*(t) \phi(t)) =0
\]
and so by integrating this from $t_1$ to $\tau$ we have
\[
     w^*(\tau)\phi(\tau)= w^*(t_1) \phi(t_1).
\]
Since $\phi(\tau)=1$ and $|\phi(t_1)| \leq (t_1/\tau)^{a}$ holds,
by applying this to the above equality we have (\ref{2.15}).
\par
   Let us show (\ref{2.16}). In this case, we set
\[
     \phi_1(t)=\exp \Bigl[ \int_t^{\tau}
          \frac{(\lambda(s,x(s))+a(s,x(s))+\ell(t,x(t)))}{s}ds 
          \Bigr], \quad t_1 \leq t \leq \tau.
\]
Since ${\rm Re}(\lambda(s,x(s))+a(s,x(s))+\ell(t,x(t)))
            <-2a+A+L<-a$ we have
$|\phi_1(t)| \leq (t/\tau)^a$ for $t_1 \leq t \leq \tau$.
Then, we can reduce (\ref{2.14}) into
\[
      \frac{d}{dt}(\phi_1(t)q^*(t))
         = \phi_1(t) \gamma(t,x(t)) w^*(t), 
\]
and so by integrating this from $t_1$ to $\tau$ and by using
(\ref{2.15}) (with $\tau$ replaced by $t$) we have
\begin{align*}
    |q^*(\tau)| &\leq |\phi(t_1)q^*(t_1)|
          + \int_{t_1}^{\tau} |\phi_1(t)\gamma(t,x(t)) w^*(t)|
            \frac{dt}{t} \\
    &\leq \Bigl( \frac{t_1}{\tau} \Bigr)^{a}|q^*(t_1)|+
            \int_{t_1}^{\tau} \Bigl( \frac{t}{\tau} \Bigr)^{a}
           \Gamma \Bigl( \frac{t_1}{t} \Bigr)^{a}|w^*(t_1)|
            \frac{dt}{t} \\
    &= \Bigl( \frac{t_1}{\tau} \Bigr)^{a}|q^*(t_1)|
        + \Bigl( \frac{t_1}{\tau} \Bigr)^{a}\Gamma|w^*(t_1)|
          \times \log (\tau/t_1).
\end{align*}
This proves (\ref{2.16}). 
\end{proof}

\par
\medskip
   {\bf Step 4.} Recall that 
$|b_1(t,x,w,q)| \leq B_0\mu(t) + B_1|w|+B_2|q|$ holds on $\Omega^*$. 
We have

\begin{lem}\label{Lemma2.8}
    Under the above situation, we have the 
following estimate for any $t_1 \in (t_\xi,t_0)$:
\begin{align*}
    |x(t_1)| \leq |\xi| &+ B_0(\varphi(t_0)- \varphi(t_1))\\
    &+\Bigl(\frac{B_1}{a} + \frac{B_2\Gamma}{a^2} \Bigr)|w^*(t_1)|
         + \frac{B_2}{a}|q^*(t_1)|.
\end{align*}
\end{lem}

\begin{proof}
    Let $t_1 \in (t_\xi,t_0)$. By (\ref{2.12}) we have
\[
      x(t_1)= \xi + \int_{t_1}^{t_0}b(\tau,x(\tau))
           \frac{d\tau}{\tau}.
\]
Since 
\begin{align*}
   &|b(\tau,x(\tau))| \\
   &\leq B_0 \mu(\tau) + B_1|w^*(\tau)|+B_2|q^*(\tau)| \\
   &\leq B_0 \mu(\tau) 
         + B_1 \Bigl( \frac{t_1}{\tau} \Bigr)^a|w^*(t_1)|
         + B_2 \Bigl( \frac{t_1}{\tau} \Bigr)^a
       \bigl( \Gamma|w^*(t_1)| \log(\tau/t_1) + |q^*(t_1)| \bigr)
\end{align*}
holds for any $\tau \in (t_1,t_0]$, we have
\begin{align}
   |x(t_1)| \leq |\xi| &+ \int_{t_1}^{t_0}
            \Bigl(B_0 \mu(\tau) 
      + B_1\Bigl( \frac{t_1}{\tau} \Bigr)^a|w^*(t_1)| \label{2.17}\\
     &+ B_2 \Bigl( \frac{t_1}{\tau} \Bigr)^a
       \bigl( \Gamma|w^*(t_1)| \log(\tau/t_1) + |q^*(t_1)| \bigr)
             \Bigr) \frac{d\tau}{\tau}. \notag
\end{align}
Here, we note:
\begin{align*}
     &\int_{t_1}^{t_0}\Bigl( \frac{t_1}{\tau} \Bigr)^a
            \frac{d\tau}{\tau}
         = \frac{1}{a} \Bigl(1- \frac{{t_1}^a}{{t_0}^a} \Bigr)
         \leq \frac{1}{a}, \\
    &\int_{t_1}^{t_0}\Bigl( \frac{t_1}{\tau} \Bigr)^a
            \log(\tau/t_1) \frac{d\tau}{\tau}
     = \frac{{t_1}^a}{-a{t_0}^a} \log(t_0/t_1) 
             + \frac{1}{a^2} \Bigl(1- \frac{{t_1}^a}{{t_0}^a} \Bigr)
              \leq \frac{1}{a^2}.
\end{align*}
By applying these estimates to (\ref{2.17}), we have 
Lemma \ref{Lemma2.8}.
\end{proof}

\begin{cor}\label{Corollary2.9}
    If $\xi \in D_{R/2}$ we have $t_{\xi}=0$.
\end{cor}

\begin{proof}
    Let $|\xi|<R/2$. Let us show that if $t_{\xi}>0$ 
holds we have a contradiction. 
\par
   Suppose that $t_{\xi}>0$ holds. Then, by Lemmas \ref{Lemma2.6} 
and \ref{Lemma2.8} we have
\[
    |x(t_1)| \leq \frac{R}{2}+B_0 \varphi(\sigma)
         +\Bigl(\frac{B_1}{2a} + \frac{B_2\Gamma}{a^2} \Bigr)r_1
         + \frac{B_2}{a}r_2 =R_1< R
\]
for any $t_1 \in (t_\xi,t_0)$. Since 
$K=\{x \in \BC^n \,;\, |x| \leq R_1 \}$ is a compact subset of 
$D_R$ and since $x(t_1) \in K$ for any $t_1 \in (t_\xi,t_0]$, 
by a theorem in ordinary differential equations (for example, by 
Theorem 4.1 in Coddington-Levinson \cite{CL}) we can extend $x(t)$ 
to $(t_{\xi}-\varepsilon, t_0]$ for some $\varepsilon>0$. This 
contradicts the condition that $(t_{\xi}, t_0]$ is the maximal 
interval of the existence of the solution $x(t)$.
\end{proof}

\par
\medskip
   {\bf Step 5.} Since $t_{\xi}=0$, by (\ref{2.15}) with $\tau=t_0$ 
we have
\[
    |w^*(t_0)| \leq \Bigl( \frac{t_1}{t_0} \Bigr)^a|w^*(t_1)|
                  \leq \Bigl( \frac{t_1}{t_0} \Bigr)^a r_1
\]
for any $t_1 \in (0,t_0)$. Since $r_1>0$ is independent of $t_1$,
by letting $t_1 \longrightarrow +0$ we have $w^*(t_0)=0$.
Since $w^*(t_0)=w(t_0,\xi)$ we have $w(t_0,x)=0$ for any
$x \in D_{R/2}$. Since $t_0 \in (0,\sigma)$ is taken arbitrarily
we have $w(t,x)=0$ on $(0,\sigma) \times D_{R/2}$.
\par
   This completes the proof of Proposition \ref{Proposition2.5}. 
\end{proof}

%
%

\subsection{Application}\label{subsection2.3}

    Let us apply Theorem \ref{Theorem2.2} to the problem 
of analytic continuation of solutions of Briot-Bouquet type
partial differential equations.
\par
   Let $(t,x)$ be the variables in $\BC_t \times \BC_x$, and let 
$F(t,x,u,v)$ be a function in a neighborhood $\Delta$ of 
the origin of $\BC_t \times \BC_x \times \BC_u \times \BC_v$. 
Set $\Delta_0=\Delta \cap \{t=0, u=0, v=0 \}$.
In this subsection, we consider the following equation
\begin{equation}\label{2.18}
   t \,\frac{\partial u}{\partial t}
   = F \Bigl( t,x,u, \frac{\partial u}{\partial x} \Bigr)
\end{equation}
(in the germ sense at $(0,0) \in \BC_t \times \BC_x$) 
under the assumptions
\par
\medskip
   ${\rm B}_1)$ \enskip $F(t,x,u,v)$ is holomorphic in $\Delta$,
\par
   ${\rm B}_2)$ \enskip $F(0,x,0,0) \equiv 0$ in
       $\Delta_0$, \enskip and
\par
   ${\rm B}_3)$ \enskip $(\partial F/\partial v)
           (0,x,0,0) \equiv 0$ in $\Delta_0$.
\par
\medskip
\noindent
Then, equation (\ref{2.18}) is called a Briot-Bouquet type partial 
differential equation with respect to $t$ (by G\'erard-Tahara
\cite{gt1,book}), and the function
\[
    \lambda (x) = \frac{\partial F}{\partial u}
           (0,x,0,0)  
\]
is called the characteristic exponent of (\ref{2.18}).
This equation was studied by \cite{gt1} and Yamazawa \cite{yamazawa}. 
\par
   By \cite{gt1} we know that if 
$\lambda(0) \not\in \{1,2,\ldots \}$ equation (\ref{2.18}) has a unique
holomorphic solution $u_0(t,x)$ in a neighborhood of 
$(0,0) \in \BC_t \times \BC_x$ satisfying $u_0(0,x)=0$ near $x=0$.
Therefore, by applying Theorem \ref{Theorem2.2} (with $\mu(t)=t$) to 
this case we have

\begin{thm}\label{Theorem2.10}
    Suppose the conditions ${\rm B}_1)$, ${\rm B}_2)$, ${\rm B}_3)$
and ${\rm Re}\lambda(0)<0$. Let $u(t,x)$ be a holomorphic solution
of {\rm (\ref{2.18})} in a neighborhood of $(0,\sigma_0) \times D_{R_0}$ 
for some $\sigma_0>0$ and $R_0>0$. If $u(t,x)$ satisfies
\begin{equation}\label{2.19}
     \varlimsup_{R \to +0} \, \biggl[ \, 
          \lim_{\sigma \to +0} \, \Bigl( \frac{1}{R^2}
       \sup_{(0,\sigma) \times D_R}|u(t,x)| \Bigr) \, \biggr]
       \, = \, 0,  
\end{equation}
$u(t,x)$ can be continued holomorphically up to a neighborhood of 
$(0,0) \in \BC_t \times \BC_x$.
\end{thm}

\begin{rem}\label{Remark2.11}
    The following example shows that we need some condition like 
(\ref{2.19}) in order to get the analytic continuation of solutions: 
the equation
\[
     t \frac{\partial u}{\partial t}
     = - 2u +xt \Bigl(\frac{\partial u}{\partial x} \Bigr)^2
\]
has a solution $u=x/t$.
\end{rem}

%
%
\section{Analysis in Case 2}\label{section3}

   Let us consider Case 2 in a little bit general setting.
We consider the equation 
\begin{equation}\label{3.1}
     t \,\frac{\partial u}{\partial t}
     = \alpha(t,x)+ \lambda(t,x)u + 
       (\beta(t,x)+xc(t,x))\frac{\partial u}{\partial x}
       +R_2 \Bigl(t,x,u,\frac{\partial u}{\partial x} \Bigr)
\end{equation}
where $\alpha(t,x)$, $\lambda(t,x)$, $\beta(t,x)$ and $c(t,x)$ are 
continuous functions on $[0,T_0] \times D_{R_0}$ that are holomorphic 
in $x$ for any fix $t$ and satisfy
\begin{align}
   &\sup_{x \in D_{R_0}}|\alpha(t,x)|= O(\mu(t)) 
               \quad \mbox{(as $t \longrightarrow +0$)}, 
           \label{3.2} \\
   &\sup_{x \in D_{R_0}}|\beta(t,x)|=O(\mu(t)) 
                 \quad \mbox{(as $t \longrightarrow +0$)}, 
          \label{3.3} \\
   &{\rm Re} \,c(t,x) \leq 0 \quad 
             \mbox{on $[0,T_0] \times D_{R_0}$}, \label{3.4}
\end{align}
and $R_2(t,x,u,v)$ is a continuous function on $\Omega$ (where 
$\Omega$ is the same as in \S 1) which is holomorphic in the variable 
$(x,u,v)$ for any fixed $t$ and has a Taylor expansion in $(u,v)$ of 
the form:
\[
        R_2(t,x,u,v)= \sum_{i+j \geq 2}
           a_{i,j}(t,x) u^iv^j.
\]
As to the existence of a solution, we know a unique solvability 
result. By [Theorem 5.1 in Bacani-Tahara \cite{BT1}] we have

\begin{thm}\label{Theorem3.1} 
     Suppose the conditions {\rm (\ref{3.2})}, {\rm (\ref{3.3})} 
and {\rm (\ref{3.4})}. If ${\rm Re}\lambda (0,0)<0$ holds, there 
are $T>0$, $R>0$ and $r>0$ such that equation {\rm (\ref{3.1})} has 
a unique solution $u_0(t,x) \in {\mathscr X}_1(W_{T,R,r})$ satisfying
\[
      |u_0(t,x)| \leq M\mu(t) \quad \mbox{and} \quad 
      \Bigl| \frac{\partial u_0}{\partial x}(t,x)
         \Bigr| \leq M\mu(t)
\]
on $W_{T,R,r}$ for some $M>0$.
\end{thm}

%
%
\subsection{Uniqueness result in Case 2}\label{subsection3.1}

   The following theorem is the main result of this section.

\begin{thm}\label{Theorem3.2} 
    Suppose the conditions {\rm (\ref{3.2})}, {\rm (\ref{3.3})}, 
{\rm (\ref{3.4})} and ${\rm Re}\lambda(0,0)<0$. Let 
$u(t,x) \in {\mathscr X}_1((0,T) \times D_R)$ be a solution of 
{\rm (\ref{3.1})} with $T>0$ and $R>0$. If $u(t,x)$ satisfies
\begin{equation}\label{3.5}
     \varlimsup_{R \to +0} \, \biggl[ \, 
          \lim_{\sigma \to +0} \, \Bigl( \frac{1}{R^2}
       \sup_{(0,\sigma) \times D_R}|u(t,x)| \Bigr) \, \biggr]
       \, = \, 0, 
\end{equation}
we have $u(t,x)=u_0(t,x)$ 
on $(0,T_1) \times D_{R_1}$ for some $T_1>0$ and $R_1>0$,
where $u_0(t,x)$ is the solution obtained in 
Theorem \ref{Theorem3.1}.
\end{thm}

\begin{cor}\label{Corollary3.3}
     Suppose the conditions {\rm (\ref{3.2})}, {\rm (\ref{3.3})}, 
{\rm (\ref{3.4})} and ${\rm Re}\lambda(0,0)<0$.
If a solution $u(t,x) \in {\mathscr X}_1((0,T) \times D_R)$ 
of {\rm (\ref{3.1})} satisfies 
\[
     \lim_{t \to +0} \,\Bigl(\sup_{x \in D_R}|u(t,x)| \Bigr)\,
       = \, 0,
\]
we have $u(t,x)=u_0(t,x)$ on $(0,T_1) \times D_{R_1}$ for 
some $T_1>0$ and $R_1>0$.
\end{cor}

\begin{rem}\label{Remark3.4}
    (1) In the case ${\rm Re}\lambda(0,0)>0$
we have the following counter example: the equation
\[
     t \, \frac{\partial u}{\partial t}
    = 2u - x \frac{\partial u}{\partial x}
       + u \Bigl( \frac{\partial u}{\partial x} \Bigr)
\]
has a trivial solution $u \equiv 0$, a nontrivial solution $u=t^2$ 
and a family of solutions
\[
    u= \frac{xt}{c-t}
\]
with an arbitrary constant $c$. 
\par
   (2) In the case ${\rm Re}\lambda(0,0)=0$, we have the following
counter example: the equation
\[
     t \, \frac{\partial u}{\partial t}
    = - x \frac{\partial u}{\partial x}+u^2
       + \Bigl( \frac{\partial u}{\partial x} \Bigr)^2
\]
has a trivial solution $u \equiv 0$ and a family of nontrivial
solutions
\[
    u= \frac{1}{c- \log t}
\]
with an arbitrary constant $c$. 
\par
   (3) In the case ${\rm Re}\lambda(0,0) <0$, the following example 
shows that the condition (\ref{3.5}) is reasonable: the equation
\[
     t \, \frac{\partial u}{\partial t}
    = -u - x \frac{\partial u}{\partial x}
       + \Bigl( \frac{\partial u}{\partial x} \Bigr)^2
\]
has a trivial solution $u \equiv 0$ and a nontrivial solution
$u=3x^2/4$.
\end{rem}

%
%
\subsection{Proof of Theorem 3.2}\label{subsection3.2}

   Since the proof of Theorem \ref{Theorem3.2} is quite similar 
to the proof of Theorem \ref{Theorem2.2}, we give here only a sketch 
of the proof.
\par
   Let $u(t,x) \in {\mathscr X}_1((0,T) \times D_R)$ be a 
solution of (\ref{3.1}) satisfying (\ref{3.5}). Set 
$w(t,x)=u(t,x)-u_0(t,x)$ where $u_0(t,x)$ is the solution obtained 
in Theorem \ref{Theorem3.1}. Then, by the same argument as in 
(\ref{2.8}) we see that $w(t,x)$ satisfies a partial 
differential equation of the form
\begin{equation}\label{3.6}
     t \frac{\partial w}{\partial t}
    - (b(t,x)+xc(t,x)) \frac{\partial w}{\partial x}
     = (\lambda(t,x)+a(t,x))w, 
\end{equation}
on $(0,\sigma_0) \times D_{R_0}$ for some $\sigma_0>0$ and
$R_0>0$. where $a(t,x)$ and $b(t,x)$ are functions belonging to
${\mathscr X}_0((0,\sigma_0)\times D_{R_0})$ that satisfy
\begin{align*}
   &\lim_{\sigma \to +0} 
         \sup_{(0,\sigma) \times D_R}|a(t,x)|= o(R) \quad
                  \mbox{(as $R \longrightarrow +0$)}, \\
   &\lim_{\sigma \to +0} 
         \sup_{(0,\sigma) \times D_R}|b(t,x)|= o(R) \quad
                   \mbox{(as $R \longrightarrow +0$)}.
\end{align*}
By applying $\partial/\partial x$ to (\ref{3.6}) we have
\begin{align}
    &t \frac{\partial q}{\partial t}
    - (b(t,x)+xc(t,x)) \frac{\partial q}{\partial x}
        \label{3.7} \\
    &= \gamma(t,x)w + (\lambda(t,x)+a(t,x)+c(t,x)+\ell(t,x))q,
         \notag
\end{align}
where 
\begin{align*}
    &\gamma(t,x)= (\partial \lambda/\partial x)(t,x)
         + (\partial a/\partial x)(t,x), \\
    &\ell(t,x)= (\partial b/\partial x)(t,x)
             + x (\partial c/\partial x)(t,x):
\end{align*}
these are also functions belonging to
${\mathscr X}_0((0,\sigma_0) \times D_{R_0})$.  If we notice the 
fact that $|x(\partial c/\partial x)(t,x)| \leq C_1|x|$ 
on $(0,\sigma_0) \times D_{R_0}$ for some $C_1>0$, 
by taking $\sigma>0$ and $R>0$ sufficiently small we have
the same conditions as in Lemma \ref{Lemma2.6}.
\par
   Now, let us consider the initial value problem:
\begin{equation}\label{3.8}
   t \,\frac{dx}{dt} = -(b(t,x)+xc(t,x)), 
     \quad x(t_0)=\xi.
\end{equation}
Let $x(t)$ be the unique solution in a neighborhood of $t=t_0$. 
Let $(t_{\xi},t_0]$ be the maximal interval of the existence of 
this solution. Set
\[
      w^*(t)=w(t,x(t)), \quad q^*(t)=q(t,x(t)).
\]
Since ${\rm Re}\,c(t,x) \leq 0$ is supposed (in (\ref{3.4})), we have 
${\rm Re}\,c(s,x(s)) \leq 0$, and so 
${\rm Re}(\lambda(s,x(s))+a(s,x(s))
         +c(s,x(s))+ \ell(s,x(s)))<-2a+A+0+L <-a$. Hence, by the 
same argument as in the proof of Theorem \ref{Theorem2.2} we can 
show the same conditions as in Lemmas \ref{Lemma2.7}, \ref{Lemma2.8} 
and Corollary \ref{Corollary2.9}. 
\par
    Thus, we have $w(t,x)=0$ on $(0,\sigma) \times D_{R/2}$
as in Step 5 in the proof of Theorem \ref{Theorem2.2}. This proves
Theorem \ref{Theorem3.2}.  \qed

%
%
\subsection{Application}\label{subsection3.3}

    Let us apply Theorem \ref{Theorem3.2} to the problem 
of analytic continuation of solutions of nonlinear totally
characteristic type partial differential equations.
\par
   Let us consider the same equation 
\begin{equation}\label{3.9}
   t \,\frac{\partial u}{\partial t}
   = F \Bigl( t,x,u, \frac{\partial u}{\partial x} \Bigr)
\end{equation}
as in (\ref{2.18}) in the complex domain $\Delta$ under ${\rm B}_1)$, 
${\rm B}_2)$ and 
\par
\medskip
   ${\rm B}_4)$ \enskip $(\partial F/\partial v)
           (0,x,0,0)= x c(x)$ with $c(0) \ne 0$.
\par
\medskip
\noindent
Then, this equation is a typical model of nonlinear totally
characteristic partial differential equations discussed by
Chen-Tahara \cite{chen1}. As in subsection 2.3 we set
$\lambda(x)=(\partial F/\partial u)(0,x,0,0)$. We write
$\BN^*=\{1,2,\ldots \}$ and $\BN=\{0,1,2,\ldots \}$.
\par
   Then, by \cite{chen1} we know the following result:
if $c(0) \not\in [0,\infty)$ and
\begin{equation}\label{3.10}
      i-c(0)j-\lambda(0) \ne 0 \quad 
      \mbox{for any $(i,j) \in \BN^* \times \BN$}
\end{equation}
hold, equation (\ref{3.9}) has a unique holomorphic solution 
$u_0(t,x)$ in a neighborhood of $(0,0) \in \BC_t \times \BC_x$ 
satisfying $u_0(0,x)=0$ near $x=0$.  Therefore, by applying 
Theorem \ref{Theorem3.2} (with $\mu(t)=t$) to this case we have

\begin{thm}\label{Theorem3.5}
    Suppose the conditions ${\rm B}_1)$, ${\rm B}_2)$, ${\rm B}_4)$, 
${\rm Re}\,c(0)<0$ and ${\rm Re}\lambda(0)<0$. Let $u(t,x)$ be a 
holomorphic solution of {\rm (\ref{3.9})} in a neighborhood of 
$(0,\sigma_0) \times D_{R_0}$ for some $\sigma_0>0$ and $R_0>0$. 
If $u(t,x)$ satisfies
\begin{equation}\label{3.11}
     \varlimsup_{R \to +0} \, \biggl[ \, 
          \lim_{\sigma \to +0} \, \Bigl( \frac{1}{R^2}
       \sup_{(0,\sigma) \times D_R}|u(t,x)| \Bigr) \, \biggr]
       \, = \, 0, 
\end{equation}
$u(t,x)$ can be continued holomorphically up to a neighborhood of 
$(0,0) \in \BC_t \times \BC_x$.
\end{thm}

\begin{rem}\label{Remark3.6}
     The following example shows that we need some condition like 
(\ref{3.11}) in order to get the analytic continuation of solutions: 
the equation
\[
     t \frac{\partial u}{\partial t}
     = - 2u -x \frac{\partial u}{\partial x}
            +2 xt \Bigl(\frac{\partial u}{\partial x} \Bigr)^2
\]
has a solution $u=x/t$.
\end{rem}

%
%
\section{Analysis in Case 3}\label{section4}

   Let us consider Case 3 in a little bit restricted setting.
Let $p \in \{1,2,3,\ldots\}$: we consider the equation 
\begin{align}
     t \,\frac{\partial u}{\partial t}
     = &\alpha(t,x) + \lambda(t,x)u +(\beta(t,x)+x^pc(t,x))
           \Bigl(x \frac{\partial u}{\partial x} \Bigr)
             \label{4.1} \\
       &+R_2 \Bigl(t,x,u, x\frac{\partial u}{\partial x} \Bigr)
                  \notag
\end{align}
where $\alpha(t,x)$, $\lambda(t,x)$, $\beta(t,x)$ and $c(t,x)$ are 
continuous functions on $[0,T_0] \times D_{R_0}$ that are holomorphic 
in $x$ for any fix $t$ and satisfy
\begin{align}
   &\sup_{x \in D_{R_0}}|\alpha(t,x)|= O(\mu(t)) 
               \quad \mbox{(as $t \longrightarrow +0$)}, 
           \label{4.2} \\
   &\sup_{x \in D_{R_0}}|\beta(t,x)|=O(\mu(t)) 
                 \quad \mbox{(as $t \longrightarrow +0$)}, 
          \label{4.3} \\
   &c(0,0) \ne 0,  \label{4.4}
\end{align}
and $R_2(t,x,u,v)$ is the same as in (\ref{3.1}).
In this case, equations of this type were studied by 
Chen-Luo-Zhang \cite{CLZ}, Luo-Chen-Zhang \cite{LCZ} and 
Bacani-Tahara \cite{BT2}.
\par
   By applying the change of variable $x \longrightarrow e^{i\theta}x$ 
in equation (\ref{4.1}) we see that $x^pc(t,x)$ is transformed into
$x^p(e^{i p \theta}c(t,e^{i \theta}x))$ and so by taking $\theta$ 
suitably we have the condition: $e^{i p \theta}c(0,0)<0$. Hence, 
without loss of generality we may assume
\begin{equation}\label{4.5}
           c(0,0)<0  
\end{equation}
from the first. For simplicity, we suppose this condition 
from now.
\par
   As to the existence of a solution, we know a unique solvability 
result. In order to state the existence result, we prepare some
notations: for $T>0$, $R>0$, $0<\theta<\pi/2p$ and $r>0$ we set
\begin{align*}
    &S = S(\theta,R)
          =\{x \in \BC \,;\, 0<|x|<R, |\arg x|<\theta \}, \\
    &d_{S}(x)= \min \bigl\{\log (R/|x|), \theta-|\arg x| \bigr\}, \\
    &W_{T,R,\theta,r}= \{(t,x) \in (0,T) \times S \,;\,
         \varphi(t)/r < d_S(x) \}.
\end{align*}
Then, by [Theorem 8.1 in \cite{BT2}] we have

\begin{thm}\label{Theorem4.1} 
    Suppose the conditions {\rm (\ref{4.2})}, {\rm (\ref{4.3})} 
and {\rm (\ref{4.5})}. 
If ${\rm Re}\lambda (0,0)<0$ holds, there are $T>0$, $R>0$, 
$0<\theta<\pi/2p$ and $r>0$ such that equation {\rm (\ref{4.1})} has 
a unique solution $u_0(t,x) \in {\mathscr X}_1(W_{T,R,\theta,r})$ 
satisfying
\[
     |u_0(t,x)| \leq M\mu(t) \quad \mbox{and} \quad 
       \Bigl|x \frac{\partial u_0}{\partial x}(t,x)
       \Bigr| \leq M\mu(t)
\]
on $W_{T,R,\theta,r}$ for some $M>0$.
\end{thm}

%
%
\subsection{Uniqueness result in Case 3}\label{subsection4.1}

   The following theorem is the main result of this section.

\begin{thm}\label{Theorem4.2}
   Suppose {\rm (\ref{4.2})}, {\rm (\ref{4.3})}, {\rm (\ref{4.5})} 
and ${\rm Re}\lambda (0,0)<0$.  Let 
$u(t,x) \in {\mathscr X}_1((0,T) \times S(\theta,R))$ be a solution 
of {\rm (\ref{4.1})} with $T>0$, $\theta>0$ and $R>0$. If $u(t,x)$ 
satisfies
\begin{equation}\label{4.6}
     \varlimsup_{\eta \to +0} \, \biggl[ \, 
          \lim_{\sigma \to +0} \, \Bigl( \frac{1}{\eta^2}
       \sup_{(0,\sigma) \times S(\eta \theta, \eta R)}
        |u(t,x)| \Bigr) \, \biggr] \, = \, 0,
\end{equation}
we have $u(t,x)=u_0(t,x)$ on $(0, T_1) \times S(\theta_1,R_1)$ 
for some $T_1>0$, $\theta_1>0$ and $R_1>0$, 
where $u_0(t,x)$ is the solution obtained in Theorem \ref{Theorem4.1}.
\end{thm}

\begin{cor}\label{Corollary4.3}
    Suppose the conditions {\rm (\ref{4.2})}, {\rm (\ref{4.3})}, 
{\rm (\ref{4.5})} and ${\rm Re}\lambda (0,0)<0$. Let 
$u(t,x) \in {\mathscr X}_1((0,T) \times S(\theta,R))$ be a solution 
of {\rm (\ref{4.1})}. If $u(t,x)$ satisfies
\[
     \lim_{t \to +0} \, \Bigl( 
            \sup_{x \in S(\theta,R)}|u(t,x)| \Bigr) \, = \, 0, 
\]
we have $u(t,x)=u_0(t,x)$ on $(0, T_1) \times S(\theta_1,R_1)$ 
for some $T_1>0$, $\theta_1>0$ and $R_1>0$.
\end{cor}

\begin{rem}\label{Remark4.4}
     (1) In the case ${\rm Re}\lambda(0,0)>0$
we have the following counter example: the equation
\[
     t \, \frac{\partial u}{\partial t}
    = 2u - x^2 \frac{\partial u}{\partial x}
       + \frac{x^2t}{(1-t)}\frac{\partial u}{\partial x} 
\]
has a trivial solution $u \equiv 0$, a nontrivial solution
$u=t^2$ and a family of solutions
\[
    u= \frac{c\,te^{-1/x}}{1-t}
\]
with an arbitrary constant $c$. In this case we have $p=1$, 
$\lambda(t,x)=2$, $c(t,x)=-1$, $\beta(t,x)=xt/(1-t)$, $R_2 \equiv 0$ 
and $\mu(t)=t$.
\par
   (2) In the case ${\rm Re}\lambda(0,0)= 0$
we have the following counter example: the equation
\[
     t \, \frac{\partial u}{\partial t}
    = - x^2 \frac{\partial u}{\partial x} +u^2 
       + \Bigl( x \frac{\partial u}{\partial x} \Bigr)^2
\]
has a trivial solution $u \equiv 0$ and a family of nontrivial
solution
\[
    u= \frac{1}{c- \log t}
\]
with an arbitrary constant $c$.
\par
    (3) We note: the equation
\[
     t \, \frac{\partial u}{\partial t}
    = -u - x^2 \frac{\partial u}{\partial x} 
       + t \Bigl( x \frac{\partial u}{\partial x} \Bigr)^2
\]
has a trivial solution $u \equiv 0$ and a nontrivial solution
$u=x/t$. This shows that even in the case ${\rm Re}\lambda(0,0)< 0$,
in order to get a uniqueness result we need some condition on the 
behavior of $u(t,x)$ (as $t \longrightarrow +0$). But, 
unfortunately the author does not know whether our assumption 
(\ref{4.6}) is reasonable or not: he has no good examples.
\end{rem}

%
%
\subsection{Proof of Theorem 4.2}\label{subsection4.2}

   Let $u(t,x) \in {\mathscr X}_1((0,\sigma_0) 
                             \times S(\theta_0,R_0))$ 
be a solution of (4.1) satisfying (4.6) (with $\theta$ and
$R$ replaced by $\theta_0$ and $R_0$, respectively). We may 
suppose: $0<\theta_0<\pi/2p$. Set 
\[
      w(t,x)= u(t,x)-u_0(t,x),
\]
where $u_0(t,x)$ is the solution obtained in Theorem \ref{Theorem4.1}.
We set $v_0(t,x)=x(\partial u_0/\partial x)(t,x)$.  By taking 
$\sigma_0$, $\theta_0$ and $R_0$ sufficiently small we may suppose 
that $u_0(t,x)$ and $v_0(t,x)$ are defined on 
$(0,\sigma_0) \times S(\theta_0,R_0)$ and satisfy 
$|u_0(t,x)| \leq M\mu(t)$ and
$|v_0(t,x)| \leq M\mu(t)$ on $(0,\sigma_0) \times S(\theta_0,R_0)$. 
Then, $w(t,x)$ satisfies 
\begin{equation}\label{4.7}
     \varlimsup_{\eta \to +0} \, \Bigl[ \, 
          \lim_{\sigma \to +0} \, \Bigl( \frac{1}{\eta^2}
       \sup_{(0,\sigma) \times S(\eta \theta_0, \eta R_0)}
          |w(t,x)| \Bigr) \, \Bigr] \, = \, 0 
\end{equation}
and a partial differential equation
\begin{align}
     t \,\frac{\partial w}{\partial t}
     = \lambda(t,x)w
           &+ (\beta(t,x)+x^pc(t,x))
         \Bigl( x \frac{\partial w}{\partial x} \Bigr)
        \label{4.8} \\
       &+a_1 \Bigl(t,x,w, x \frac{\partial w}{\partial x} \Bigr)w
       +b_1 \Bigl(t,x,w, x \frac{\partial w}{\partial x} \Bigr)
         \Bigl( x \frac{\partial w}{\partial x} \Bigr) \notag
\end{align}
where $a_1(t,x,w,q)$ and $b_1(t,x,w,q)$ are suitable functions 
satisfying
\begin{align*}
    &a_1(t,x,w,q)w + b_1(t,x,w,q)q \\
    &= R_2(t,x,w+u_0(t,x), q+v_0(t,x))
              - R_2(t,x,u_0(t,x),v_0(t,x)) \bigr).
\end{align*}
We may suppose that $a_1(t,x,w,q)$ and $b_1(t,x,w,q)$ belong to 
${\mathscr X}_0(\Omega_0)$ with $\Omega_0=[0,\sigma_0] \times
         S(\theta_0,R_0) \times D_{\rho_1} \times D_{\rho_1}$ 
for some $\rho_1>0$.  In addition, we have the properties:
\begin{align*}
    &|\beta(t,x)| \leq B \mu(t) \quad
             \mbox{on $(0,\sigma_0) \times D_{R_0}$}, \\
    &|a_1(t,x,w,q)| \leq A_0\mu(t)+A_1|w|+A_2|q|
        \quad \mbox{on $\Omega_0$}, \\
    &|b_1(t,x,w,q)| \leq B_0\mu(t)+B_1|w|+B_2|q|
        \quad \mbox{on $\Omega_0$}
\end{align*}
for some $B>0$, $A_i>0$ ($i=0,1,2$) and $B_i>0$ ($i=0,1,2$). Without 
loss of generality we may suppose
\[
     {\rm Re}\lambda(t,x) < -2a \quad
           \mbox{on $[0,\sigma_0] \times D_{R_0}$}
\]
for some $a>0$.  Recall that we have supposed $c(0,0)<0$. Thus, 
to prove Theorem \ref{Theorem4.2} it is sufficient to show the 
following result.

\begin{prop}\label{Proposition4.5}
    In the above situation, we have $w(t,x)=0$ on 
$(0,T_1) \times S(\theta_1,R_1)$ for some $T_1>0$, $\theta_1>0$ 
and $R_1>0$.
\end{prop}

   Before the proof, we note

\begin{lem}\label{Lemma4.6}
    If a holomorphic function $f(x)$ on $S(\theta,R)$ satisfies
\[
     \sup_{S(\eta \theta, \eta R)}|f(x)|
      = o(\eta^m) \quad 
       \mbox{{\rm (}as $\eta \longrightarrow +0${\rm )}}
\]
for some $m \geq 1$, we have
\[
     \sup_{S(\eta \theta, \eta R)} 
           |x(d/dx)f(x)|
      = o(\eta^{m-1}) \quad 
         \mbox{{\rm (}as $\eta \longrightarrow +0${\rm )}}.
\]
\end{lem}

\begin{proof}
   By the assumption, for any $\epsilon>0$ there
is an $\eta_0 \in (0,1)$ such that 
\[
       |f(x)| \leq \epsilon \eta^m \quad
       \mbox{on $S(\eta \theta, \eta R)$}, 
            \quad 0<\eta<\eta_0.
\]
Take any $0<\eta<\eta_0$ and fix it. 
Set $d(x)= \min \{\eta \theta-|\arg x|, \log (\eta R)-\log|x| \}$
for $x \in S(\eta \theta, \eta R)$. Then, by Nagumo's lemma in 
a sectorial domain (see [Lemma 4.2 in \cite{BT2}]) we have
\[
       |x (d/dx)f(x)|
        \leq \frac{\epsilon \eta^m}{d(x)} \quad
       \mbox{on $S(\eta \theta, \eta R)$}.
\]
If $x \in S((\eta/2) \theta, (\eta/2) R)$ we have 
\begin{align*}
   &\eta \theta-|\arg x|> \eta \theta-(\eta/2)\theta
          =(\eta/2) \theta
            \geq \min\{(\eta/2) \theta, \log 2 \}, \\
   &\log (\eta R)-\log|x| \geq \log(\eta R)-\log((\eta/2)R)
         = \log 2
            \geq \min\{(\eta/2) \theta, \log 2 \}
\end{align*}
and so $d(x) \geq \min\{(\eta/2) \theta, \log 2 \}$.
If $\eta>0$ is sufficiently small we have
$d(x) \geq (\eta/2) \theta$, and so
\[
        |x (d/dx)f(x)|
        \leq \frac{\epsilon \eta^m}
              {(\eta/2)\theta} = \frac{2^m \epsilon}{\theta}
        (\eta/2)^{m-1}
       \quad
       \mbox{on $S((\eta/2) \theta, (\eta/2) R)$}.
\]
This proves the result in Lemma \ref{Lemma4.6}.
\end{proof}

\begin{proof}[Proof of Proposition 4.5]
 Let us prove Proposition \ref{Proposition4.5} step by step.

\par
\medskip
   {\bf Step 1.} We set 
$q(t,x)= x(\partial w/\partial x)(t,x)$, and
\begin{align*}
    &a(t,x)=a_1(t,x,w(t,x),q(t,x)), \\
    &b(t,x)=\beta(t,x)+ b_1(t,x,w(t,x),q(t,x)):
\end{align*}
we may suppose that these functions belong to
${\mathscr X}_0((0,\sigma_0)\times S(\theta_0,R_0))$. By 
(\ref{4.8}) we have the relation
\begin{equation}\label{4.9}
     t \frac{\partial w}{\partial t}
    - x(b(t,x)+x^pc(t,x)) \frac{\partial w}{\partial x}
     = (\lambda(t,x)+a(t,x))w.
\end{equation}
By applying $x(\partial/\partial x)$ to (\ref{4.9}) we have
\begin{align}
    &t \frac{\partial q}{\partial t}
        - x(b(t,x)+x^pc(t,x)) \frac{\partial q}{\partial x} 
               \label{4.10} \\
    &= \gamma(t,x)w + (\lambda(t,x)+a(t,x)
           + \ell(t,x))q, \notag
\end{align}
where 
\begin{align*}
    &\gamma(t,x)= x(\partial \lambda/\partial x)(t,x)
         + x(\partial a/\partial x)(t,x), \\
    &\ell(t,x)= x(\partial b/\partial x)(t,x)
             + x (\partial (x^pc)/\partial x)(t,x):
\end{align*}
these are also functions belonging to
${\mathscr X}_0((0,\sigma_0) \times S(\theta_0,R_0))$.
For $0<\sigma_1<\sigma_0$ and $0<\eta<1$ we set
\begin{align*}
   &A = \sup_{(0,\sigma_1) \times 
                   S(\eta \theta_0, \eta R_0)}|a(t,x)|, \\
   &\Gamma =  \sup_{(0,\sigma_1) \times 
                  S(\eta \theta_0, \eta R_0)} |\gamma(t,x)|, \\
   &L =  \sup_{(0,\sigma_1) \times 
                     S(\eta \theta_0, \eta R_0)} |\ell(t,x)|.
\end{align*}
We set also
\[
    r_1= \sup_{(0,\sigma_1) \times 
                      S(\eta \theta_0, \eta R_0)}|w(t,x)|,\quad
    r_2= \sup_{(0,\sigma_1) \times 
                      S(\eta \theta_0, \eta R_0)}|q(t,x)|.
\]
By (\ref{4.7}) and by the same arument as in the proof of Lemma 
\ref{Lemma2.6} we have

\begin{lem}\label{Lemma4.7}
    By taking $\sigma_1>0$ and $\eta>0$ sufficiently small we have 
the following conditions: $A+L<a$,  
\[
   \delta=(B+B_0) \varphi(\sigma_1) + \Bigl( \frac{B_1}{a} 
          + \frac{B_2\Gamma}{a^2} \Bigr)r_1
         + \frac{B_2}{a}r_2 < \log 2,
\]
and $0<\sin^{-1}(2\delta) < \min \{\eta \theta_0/12, \pi/6p \}$.
\end{lem}

\par
\medskip
   {\bf Step 2.} We take $\sigma_1>0$ and $\eta>0$ as in 
Lemma \ref{Lemma4.7}, and fix them. After that, we take 
$0<\sigma <\sigma_1$ and $0<R <\eta R_0$ sufficiently small so that
\begin{equation}\label{4.11}
    \epsilon_1= \sup_{(0,\sigma) \times S(\eta \theta_0, R)}
        |\arg(-c(t,x))|< \min \{p (\eta \theta_0)/6, \pi/6 \}.
\end{equation}
Since $\arg(-c(0,0))=0$ holds, this is possible.
\par
   We take such $\sigma>0$ and $R>0$ and fix them. Set 
$\theta=\eta \theta_0$. Then, we have 
$\epsilon_1/p < \min \{\theta/6, \pi/6p \}$.

\par
\medskip
   {\bf Step 3.} Take any $t_0 \in (0,\sigma)$ and 
$\xi \in S(\theta,R)$; for a while we fix them.
\par
   Let us consider the initial value problem
\begin{equation}\label{4.12}
   t \,\frac{dx}{dt} = -x(b(t,x)+x^pc(t,x)), 
     \quad x(t_0)=\xi. 
\end{equation}
Here, we regard $b(t,x)$ and $c(t,x)$ as functions in
${\mathscr X}_0((0,\sigma) \times S(\theta,R))$.  Let $x(t)$ 
be the unique solution in a neighborhood of $t=t_0$. 
Let $(t_{\xi},t_0]$ be the maximal interval of the existence of 
this solution.  Set
\[
      w^*(t)=w(t,x(t)), \quad q^*(t)=q(t,x(t)).
\]

\begin{lem}\label{Lemma4.8}
   {\rm (1)} We have $x(t) \ne 0$ on $(t_{\xi},t_0]$.
\par
   {\rm (2)} For any $(t_1,\tau)$ satisfying $t_\xi<t_1<\tau \leq t_0$
we have
\begin{align}
    &|w^*(\tau)| \leq \Bigl( \frac{t_1}{\tau} \Bigr)^a|w^*(t_1)|, 
               \label{4.13}\\
    &|q^*(\tau)| \leq \Bigl( \frac{t_1}{\tau} \Bigr)^a
       \bigl( \Gamma|w^*(t_1)| \log(\tau/t_1) + |q^*(t_1)| \bigr).
            \label{4.14}
\end{align}
\par
   {\rm (3)} For any $t_1 \in (t_\xi,t_0]$ we have
\[
    \Bigl| \int_{t_1}^{t_0} b(\tau,x(\tau)) \frac{d\tau}{\tau}
        \Bigr| \leq \delta   
\]
where $\delta$ is the one in Lemma \ref{Lemma4.7}.
\end{lem}

\begin{proof}
    If $x(t_1)=0$ holds for 
some $t_1 \in (t_{\xi},t_0]$, $x(t)$ is a solution of
\[
    t \frac{dx}{dt}= -x (b(t,x)+x^pc(t,x)),
    \quad x(t_1)=0.
\]
Since $x \equiv 0$ is also a solution of this initial value 
problem, by the uniqueness of the solution we have
$x(t) \equiv 0$ and so $\xi=x(t_0)=0$. This contradicts the
condition $\xi \in S(\theta,R)$ (this means $\xi \ne 0$).
This proves (1).
\par
   By applying the same argument as in the proof of 
Lemma \ref{Lemma2.7} to (\ref{4.9}) and (\ref{4.10}) we have the 
estimates in (2).  By using (\ref{4.13}) and (\ref{4.14}) we can show
\begin{align*}
    &\Bigl| \int_{t_1}^{t_0} b(\tau,x(\tau)) \frac{d\tau}{\tau}
        \Bigr|  \\
    &\leq (B+B_0) (\varphi(t_0)-\varphi(t_1))
       	+\Bigl(\frac{B_1}{a} + \frac{B_2\Gamma}{a^2} \Bigr)r_1
         + \frac{B_2}{a} r_2 
\end{align*}
in the same way as in the proof of Lemma \ref{Lemma2.8}.  Therefore, 
by combining this with Lemma \ref{Lemma4.7} we have the result (3). 
\end{proof}

\begin{lem}\label{Lemma4.9}
    We set
\[
      \phi(t)= \exp \Bigl[ - \int_t^{t_0} b(\tau,x(\tau)) 
         \frac{d\tau}{\tau} \Bigr], \quad t_{\xi}<t<t_0.
\]
Then, we have $1/2 \leq |\phi(t)| \leq 2$ on $(t_{\xi},t_0]$
and
\begin{equation}\label{4.15}
    \theta_{\phi} = \sup_{(t_{\xi},t_0]} |\arg \phi(t)|
        < \min\{\theta/12, \pi/6p \}. 
\end{equation}
\end{lem}

\begin{proof}
   By (3) of Lemma \ref{Lemma4.8} and the condition
$\delta<\log 2$ (by Lemma \ref{Lemma4.7}) we have 
$|\phi(t)| \leq e^{\delta} < e^{\log 2}=2$.
Similarly, we have $1/|\phi(t)| \leq e^{\delta} \leq 2$. This 
proves the first part. Since
\[
   |\phi(t)-1| \leq \sum_{m \geq 1} \frac{1}{m!}
         \Bigl| \int_t^{t_0} b(\tau,x(\tau)) \frac{d\tau}{\tau}
        \Bigr|^m
    \leq \sum_{m \geq 1}\frac{\delta^m}{m!} 
    \leq \delta \sum_{m \geq 0}\frac{\delta^m}{m!}
       = \delta e^{\delta} < 2\delta
\]
we have $\phi(t) \in \{z \in \BC \,;\, |z-1|<2\delta \}$:
this yields $\sin |\arg \phi(t)| <2 \delta$. Hence, we have 
$\sin \theta_{\phi} \leq 2\delta$, that is, 
$\theta_{\phi} \leq \sin^{-1}(2\delta)$. By Lemma \ref{Lemma4.7} 
and $\theta=\eta \theta_0$ (in Step 2) we have
     $\theta_{\phi}<\min\{\theta/12, \pi/6p \}$.
This proves (\ref{4.15}). 
\end{proof}

\par
\medskip
   {\bf Step 4.} Let $t_{\xi}<t_1 <t_0$. By (\ref{4.12}) we have
\[
    t \frac{d}{dt}(\phi(t)x(t))
             = -(\phi(t)x(t))^{p+1} \frac{c(t,x(t))}{\phi(t)^p}.
\]
Since $x(t) \ne 0$ on $(t_{\xi},t_0]$, we have
\[
    \frac{d}{dt} \Bigl( \frac{-1/p}{(\phi(t)x(t))^p} \Bigr)
    = - \frac{c(t,x(t))}{\phi(t)^p} \times \frac{1}{t}
\]
and so by integrating this from $t_1$ to $t_0$ we have
\[
   \frac{-1/p}{(\phi(t_0)x(t_0))^p}
                   - \frac{-1/p}{(\phi(t_1)x(t_1))^p}
      = - \int_{t_1}^{t_0} \frac{c(\tau,x(\tau))}{\phi(\tau)^p}
         \frac{d \tau}{\tau},
\]
that is, 
\[
   \frac{1}{(\phi(t_1)x(t_1))^p}
      = \frac{1}{\xi^p}- p\int_{t_1}^{t_0}
               \frac{c(\tau,x(\tau))}{\phi(\tau)^p}
         \frac{d \tau}{\tau}.
\]
Hence, by solving $x(t_1)$ we have the expression:
\begin{equation}\label{4.16}
    x(t_1)
      = \dfrac{\xi/\phi(t_1)}{\displaystyle 
           \Bigl(1- p \xi^p\int_{t_1}^{t_0} 
           \frac{c(\tau,x(\tau))}{\phi(\tau)^p}
         \frac{d \tau}{\tau} \Bigr)^{1/p}}, \quad
     t_{\xi}<t_1 \leq t_0.
\end{equation}

\begin{lem}\label{Lemma4.10}
    We have the following properties.
\par
   {\rm (1)} For any $t_1 \in (t_{\xi},t_0]$ we have 
$$
      |\xi/\phi(t_1)| \leq 2|\xi| \quad \mbox{and}
    \quad |\arg(\xi/\phi(t_1))| \leq |\arg \xi|+\theta_{\phi}.
$$
\par
   {\rm (2)} If $p|\arg \xi| + \epsilon_1 + p\theta_{\phi} \leq \pi/2$, 
we have
$$
    \biggl|\arg \biggl(- p \xi^p\int_{t_1}^{t_0} 
           \frac{c(\tau,x(\tau))}{\phi(\tau)^p}
         \frac{d \tau}{\tau} \biggr) \biggr|
     \leq p |\arg \xi| + \epsilon_1 + p\theta_{\phi}.
$$
\par
   {\rm (3)} If $p|\arg \xi| + \epsilon_1 + p\theta_{\phi} \leq \pi/2$,
for any $t_1 \in (t_{\xi},t_0]$ we have 
\begin{align}
   &\frac{|\xi|/2}
       {\Bigl(1+p|\xi|^p C_0 2^p \log(t_0/t_1) \Bigr)^{1/p}}
          \leq |x(t_1)| \leq 2|\xi|, \label{4.17}\\
   &|\arg x(t_1)| \leq 2|\arg \xi| + 2\theta_{\phi} +\epsilon_1/p,
          \label{4.18}
\end{align}
where $C_0$ is a constant satisfying $|c(t,x)| \leq C_0$ on 
$(0,\sigma) \times S(\theta,R)$.
\end{lem}

\begin{proof}
    (1) follows from Lemma \ref{Lemma4.9}. By (\ref{4.11}) and 
(\ref{4.15}) we have $|\arg(-c(t,x))| \leq \epsilon_1$ and 
$|\arg(1/\phi(t)^p)| \leq p \theta_{\phi}$. Therefore, we have
$$
    \Bigl|\arg \Bigl(- p \xi^p 
           \frac{c(\tau,x(\tau))}{\phi(\tau)^p} \Bigr) \Bigr|
     \leq p |\arg \xi| + \epsilon_1 + p\theta_{\phi}.
$$
If $p|\arg \xi| + \epsilon_1 + p\theta_{\phi} \leq \pi/2$ holds, 
the set $\{z \in \BC \setminus \{0\} \,;\, |\arg z| \leq 
            p|\arg \xi| + \epsilon_1 + p\theta_{\phi} \}$ is 
closed with respect to the addition. This proves (2).
\par
   Let us show (3). We know that $|\xi/\phi(t_1)| \geq |\xi|/2$.
Since
\begin{align*}
    &\Bigl|1- p \xi^p\int_{t_1}^{t_0} 
           \frac{c(\tau,x(\tau))}{\phi(\tau)^p}
         \frac{d \tau}{\tau} \Bigr) \Bigr| 
         \leq 1+p|\xi|^p \int_{t_1}^{t_0}
         \frac{|c(\tau,x(\tau))|}{|\phi(\tau)|^p}
         \frac{d \tau}{\tau} \\
    &\leq  1+p|\xi|^p \int_{t_1}^{t_0}
         C_0 2^p \frac{d \tau}{\tau}
      = 1+p|\xi|^p C_0 2^p \log(t_0/t_1),
\end{align*}
we have the first inequality of (\ref{4.17}).
\par
   If $p|\arg \xi| + \epsilon_1 + p\theta_{\phi} \leq \pi/2$
holds, by (2) we have
\begin{equation}\label{4.19}
    {\rm Re} \Bigl(- p \xi^p\int_{t_1}^{t_0} 
           \frac{c(\tau,x(\tau))}{\phi(\tau)^p}
         \frac{d \tau}{\tau} \Bigr) \geq 0 
\end{equation}
and so we have
\[
    {\rm Re} \Bigl(1- p \xi^p\int_{t_1}^{t_0} 
           \frac{c(\tau,x(\tau))}{\phi(\tau)^p}
         \frac{d \tau}{\tau} \Bigr) \geq 1
\]
which yields
\[
    \biggl|\Bigl(1- p \xi^p\int_{t_1}^{t_0} 
           \frac{c(\tau,x(\tau))}{\phi(\tau)^p}
         \frac{d \tau}{\tau} \Bigr)^{1/p} \biggr| \geq 1.
\]
By combining this with (1) we have $|x(t_1)| \leq 2|\xi|$.
\par
   Similarly, by (\ref{4.19}) and the result (2) we have
\[
    \biggl|\arg \Bigl(1- p \xi^p\int_{t_1}^{t_0} 
           \frac{c(\tau,x(\tau))}{\phi(\tau)^p}
         \frac{d \tau}{\tau} \Bigr) \biggr|
      \leq p|\arg \xi| + \epsilon_1 + p\theta_{\phi}.
\]
Hence, we have
\begin{align*}
   &|\arg x(t_1)| \\
   &\leq |\arg \xi| + |\arg \phi(t_1)|
     + \frac{1}{p} \biggl| \arg\Bigl(1- p \xi^p\int_{t_1}^{t_0} 
           \frac{c(\tau,x(\tau))}{\phi(\tau)^p}
         \frac{d \tau}{\tau} \Bigr) \biggr| \\
     &\leq |\arg \xi| + \theta_{\phi} +\frac{1}{p}
         (p|\arg \xi| + \epsilon_1 + p\theta_{\phi})
      = 2|\arg \xi| + 2\theta_{\phi} +\epsilon_1/p.
\end{align*}
This proves (\ref{4.18}). 
\end{proof}

\par
\medskip
   {\bf Step 5.} We recall that $0<\theta<\theta_0 <\pi/2p$
holds. By summing up we have

\begin{lem}\label{Lemma4.11}
    If $\xi \in S(\theta/3,R/3)$ we have $t_{\xi}=0$.
\end{lem}

\begin{proof}
    Let $\xi \in S(\theta/3,R/3)$. Suppose that 
$t_{\xi}>0$, and let us derive a contradiction. We note:
\begin{align*}
    p|\arg \xi| + \epsilon_1 + p\theta_{\phi}
    &< p (\theta/3)+  p\min\{\theta/6, \pi/6p \}
             +p \min\{\theta/12, \pi/6p \} \\
    &< p (\pi/6p)+  p(\pi/6p) +p(\pi/6p) =\pi/2.
\end{align*}
Therefore, by (3) of Lemma \ref{Lemma4.10} we have
\begin{align}
   &R_1=\frac{|\xi|/2}
       {\Bigl(1+p|\xi|^p C_0 2^p \log(t_0/t_{\xi}) \Bigr)^{1/p}}
          \leq |x(t_1)| \leq 2|\xi|<2R/3, \label{4.20}\\
   &|\arg x(t_1)| 
      \leq 2|\arg \xi| + 2\theta_{\phi} +\epsilon_1/p 
    \leq 2(\theta/3) + 2\theta_{\phi} +\epsilon_1/p.
            \label{4.21}
\end{align}
If we set $\theta_1=2(\theta/3) + 2\theta_{\phi} +\epsilon_1/p$,
we have $\theta_1<2(\theta/3) + 2(\theta/12) +\theta/6=\theta$ and so
we see that 
the set $K=\{x \in S(\theta,R) \,;\, R_1 \leq |x| \leq 2R/3,
   |\arg x| \leq \theta_1 \}$ is a compact subset of $S(\theta,R)$.
\par
   By (\ref{4.20}) and (\ref{4.21}) we have $x(t_1) \in K$ for any 
$t_1 \in (t_{\xi},t_0]$. Therefore, we can conclude that $x(t)$ can 
be extended to an interval $(t_{\xi}-\varepsilon,t_0]$ for some 
$\varepsilon>0$. This contradicts the condition that $(t_{\xi},t_0]$ 
is a maximal interval of the existence of the solution $x(t)$.  
\end{proof}

\par
\medskip
   {\bf Step 6.} Since $t_{\xi}=0$, by (\ref{4.13}) with $\tau=t_0$ 
we have
\[
    |w^*(t_0)| \leq \Bigl( \frac{t_1}{t_0} \Bigr)^a|w^*(t_1)|
                  \leq \Bigl( \frac{t_1}{t_0} \Bigr)^a r_1
\]
for any $t_1 \in (0,t_0)$. Since $r_1>0$ is independent of $t_1$,
by letting $t_1 \longrightarrow +0$ we have $w^*(t_0)=0$.
Since $w^*(t_0)=w(t_0,\xi)$ we have $w(t_0,x)=0$ for any
$x \in S(\theta/3,R/3)$. Since $t_0 \in (0,\sigma)$ is taken 
arbitrarily we have $w(t,x)=0$ on 
$(0,\sigma) \times S(\theta/3,R/3)$.
\par
   This completes the proof of Proposition \ref{Proposition4.5}
\end{proof}

\par
\vspace{2mm}
%
%

%
%
%
%
%
\end{document}